\documentclass[12pt]{amsart}
\usepackage{amssymb}

\newtheorem{theorem}{Theorem}[section]
\newtheorem{lemma}[theorem]{Lemma}
\newtheorem{proposition}[theorem]{Proposition}

\theoremstyle{definition}
\newtheorem{definition}[theorem]{Definition}

\numberwithin{equation}{section}

\begin{document}

\baselineskip=15pt

\title[Semidirect products and invariant connections]{Semidirect products and
invariant connections}

\author[I. Biswas]{Indranil Biswas}

\address{School of Mathematics, Tata Institute of Fundamental
Research, Homi Bhabha Road, Bombay 400005, India}

\email{indranil@math.tifr.res.in}



\thanks{Supported by a J. C. Bose fellowship.}

\date{}

\begin{abstract}
Let $S$ be a complex reductive group acting holomorphically on a complex Lie group
$N$ via holomorphic
automorphisms. Let $K(S)\,\subset\, S$ be a maximal compact subgroup. The
semidirect product $G\, :=\, N\rtimes K(S)$ acts on $N$ via biholomorphisms. We give 
an explicit description of the isomorphism classes of $G$--equivariant almost
holomorphic hermitian principal bundles on $N$. Under the assumption that there is
a central subgroup $Z\,=\, \text{U}(1)$ of $K(S)$ that acts on $\text{Lie}(N)$ as
multiplication through a single nontrivial character, we give an
explicit description of the isomorphism classes of $G$--equivariant
holomorphic hermitian principal bundles on $N$.
\end{abstract}

\maketitle

\section{Introduction}\label{sec1}

Let $N$ and $S$ be connected complex Lie groups with $S$ acting holomorphically on
$N$ via automorphisms. The semidirect product $N\rtimes S$ acts
holomorphically on the complex manifold $N$. Our starting point is the
observation that the holomorphic principal $S$--bundle
$$
N\rtimes S\,\longrightarrow\, (N\rtimes S)/S\,=\, N
$$
has a tautological flat holomorphic connection.

Assuming that $S$ is reductive, fix a maximal compact subgroup $K(S)$ of it, and define
$$
G\, :=\, N\rtimes K(S)\, \subset\, N\rtimes S\, .
$$

Let $H$ be a complex connected reductive group and $K$ a maximal compact subgroup of it.
A hermitian structure on a principal $H$--bundle $E_H$ over $N$ is a $C^\infty$ reduction
of structure group of $E_H$ to $K$.

Our aim here is to study the $G$--equivariant (almost) holomorphic hermitian principal
$H$--bundles on $N$.

Take a homomorphism $\beta\, :\, K(S)\,\longrightarrow\, K$.
The action of $K(S)$ on $N$ produces an action of
$K(S)$ on the Lie algebra ${\mathfrak n}$ of $N$.
Using $\beta$, the adjoint action of $K$ on its Lie algebra $\mathfrak k$
produces an action of $K(S)$ on $\mathfrak k$. Combining these, we get an
action of $K(S)$ on
$$
{\mathcal W}\, :=\, \text{Hom}_{\mathbb R}({\mathfrak n}\, ,{\mathfrak k})\, .
$$

Consider all pairs $(\beta\, ,\omega)$, where
\begin{itemize}
\item $\beta\, :\, K(S)\,\longrightarrow\, K$ is a homomorphism, and

\item $\omega\,\in\, {\mathcal W}^{K(S)}\,\subset\, \mathcal W$ is an invariant.
\end{itemize}
Two such pairs $(\beta\, ,\omega)$ and $(\beta'\, ,\omega')$
are called equivalent if there is an element $k\,\in\, K$ such that
\begin{itemize}
\item $\beta'(g)\,=\, k\beta(g)k^{-1}$ for all $g\,\in\, K(S)$, and

\item $\omega'(v)\, =\, \text{Ad}(k)((\omega)(v))$ for all $v\,\in\,
\mathfrak n$.
\end{itemize}
Let $\mathcal C$ denote the set of equivalence classes of all such pairs.

We prove the following (see Lemma \ref{lem2}):

\begin{lemma}\label{lei}
The set of isomorphism classes of equivariant almost holomorphic hermitian
principal $H$--bundle over $N$ is in bijection with $\mathcal C$.
\end{lemma}

The proof of Lemma \ref{lei} uses the tautological flat connection mentioned at the
beginning. The bijection in Lemma \ref{lei} is very explicit. It is described in the
proof of Theorem \ref{thm1}. A related result is proved in \cite{BT}.

Now assume that there is a central subgroup $Z\,=\, \text{U}(1)$ of $K(S)$ that acts
on the Lie algebra $\mathfrak n$ as multiplication through a single nontrivial
character of $Z$.

Since ${\mathfrak h}\,=\, {\mathfrak k}\otimes {\mathbb C}$ (recall that $H$
is reductive), we have $\mathcal W\,=\, \text{Hom}_{\mathbb C}(\overline{\mathfrak n}\, ,
{\mathfrak h})$, where $\overline{\mathfrak n}$ is the conjugate of
$\mathfrak n$. Any $\mathbb C$--linear map
$$
\alpha\, :\, \overline{\mathfrak n}\,\longrightarrow\,\mathfrak h
$$
produces a linear map $\bigwedge^2 \alpha\, :\, \bigwedge^2 \overline{\mathfrak n}\,
\longrightarrow\,\bigwedge^2 \mathfrak h$. Composing $\bigwedge^2 \alpha$ with the
Lie bracket $\bigwedge^2 \mathfrak h\, \longrightarrow\, \mathfrak h$, we get
$$
\varphi(\alpha)\, :\, \bigwedge\nolimits^2 \overline{\mathfrak n}\,\longrightarrow\,
\mathfrak h\, .
$$
Define
$$
{\mathcal C}_0\, :=\, \{(\beta\, ,\omega)\,\in\, {\mathcal C}\,\mid\,
\varphi(\omega)\,=\, 0\}\,\subset\, \mathcal C\, .
$$

We prove the following (see Theorem \ref{thm1}):

\begin{theorem}\label{thm0}
The set of isomorphism classes of equivariant holomorphic hermitian principal
$H$--bundle over $N$ is in bijection with ${\mathcal C}_0$.
\end{theorem}

\section{Semidirect products and a tautological connection}\label{se2.1}

\subsection{Action of a semidirect product}

Let $N$ be a connected complex Lie group. The group of all holomorphic
automorphisms of the group $N$ that are
connected to the identity map of $N$ will be denoted by $\text{Aut}(N)$.
In other words, $\text{Aut}(N)$ is the connected component, containing the identity
element, of the group of holomorphic automorphisms of $N$. So $\text{Aut}(N)$
is a connected complex Lie group. Let $S$ be a connected complex
affine algebraic group and
\begin{equation}\label{eta}
\eta\, :\, S\,\longrightarrow\, \text{Aut}(N)
\end{equation}
a holomorphic homomorphism of Lie groups. Let $N\rtimes S$ be the
corresponding semidirect product. The underlying set for $N\rtimes S$
is $N\times S$, and the group structure on it is defined by the rule
$$
(u_1\, ,g_1)\cdot (u_2\, ,g_2)\,=\, (u_1\eta(g_1)(u_2)\, ,g_1g_2)\, .
$$

The complex Lie group $N\rtimes S$ acts on the complex manifold $N$ as follows:
\begin{equation}\label{e1}
(u\, ,g)(v)\,=\, u\eta (g)(v)\, ,
\end{equation}
where $(u\, ,g)\,\in\, N\times S$ and $v\,\in\, N$. This action is clearly holomorphic;
it does not preserve the group structure of $N$.

\subsection{A connection}\label{sec.conn}

Consider the complex subgroup $S\, \subset\, N\rtimes S$ given by the subset
$\{e_N\}\times S\, \subset\, N\times S$, where $e_N$ is the identity element
of $N$. The projection
\begin{equation}\label{pb}
N\rtimes S\,\longrightarrow\, (N\rtimes S)/S
\end{equation}
is a holomorphic principal $S$--bundle. Since
$$
(u\, ,g)(e_N\, , g_1)\,=\, (u\, ,gg_1)\, ,
$$
for all $u\,\in\, N$ and $g\, ,g_1\,\in\, S$, the projection
$N\rtimes S\,\longrightarrow\, N$ defined by $(u\, ,g)\,\longmapsto\, u$
factors through the quotient $(N\rtimes S)/S$. The resulting map
$$
(N\rtimes S)/S\,\longrightarrow\, N
$$
is a biholomorphism. We will show that the holomorphic principal $S$--bundle
in \eqref{pb} has a natural flat holomorphic connection.

Let $\mathfrak n$ and $\mathfrak p$ be the Lie algebras of $N$ and $N\rtimes S$
respectively. Since $N$ is a normal subgroup of $N\rtimes S$, we conclude that
$\mathfrak n$ is an ideal of $\mathfrak p$. The holomorphic tangent bundle
of $N\rtimes S$ will be denoted by $T^{1,0}(N\rtimes S)$. Let
\begin{equation}\label{cH}
{\mathcal H}\,\subset\, T^{1,0}(N\rtimes S)
\end{equation}
be the holomorphic subbundle obtained by translating the above mentioned subspace
$\mathfrak n\, \subset\, \mathfrak p$ using the left--translation action of
$N\rtimes S$ on itself. Since $\mathfrak n$ is an ideal in $\mathfrak p$, the
right--translation action of $S$ on $N\rtimes S$ preserves this subbundle $\mathcal H$.
It can be shown that $\mathcal H$ is a direct summand of the
holomorphic subbundle of $T^{1,0}(N\rtimes S)$
given by the orbits of the right--translation action of $S$ on $N\rtimes S$. Indeed, the two
subbundles of $T^{1,0}(N\rtimes S)$ are clearly transversal at the identity
element. Since both the subbundles are preserved by the
left--translation action of $N\rtimes S$ on itself, they are transversal everywhere.

Since $\mathcal H$ is preserved by the right--translation action of $S$ on $N\rtimes S$,
and it is a direct summand of the holomorphic subbundle of $T^{1,0}(N\rtimes S)$
given by the orbits of the right--translation action of $S$, there is a
connection on the principal $S$--bundle in \eqref{pb} whose horizontal distribution
coincides with $\mathcal H$. Clearly, this condition uniquely determines the connection.
The connection on the principal $S$--bundle in \eqref{pb}
constructed this way will be denoted by $\nabla^S$.

Since $\mathfrak n$ is closed under the Lie bracket operation on $\mathfrak p$, the
distribution $\mathcal H$ in \eqref{cH} is integrable. Therefore, the above connection
$\nabla^S$ is flat. The connection $\nabla^S$ is holomorphic because the distribution
$\mathcal H\,\subset\, T^{1,0}(N\rtimes S)$ is holomorphic.

\section{Equivariant holomorphic hermitian bundles}\label{se3.1}

Henceforth, we assume that the group $S$ is reductive complex linear algebraic group.

Any two maximal compact subgroups of $S$ are conjugate by an element
of $S$ (see \cite[p. 256, Theorem 2.2(ii)]{He}). Fix a maximal compact subgroup
$$
K(S)\, \subset\, S\, .
$$
Define the subgroup
\begin{equation}\label{h1}
G\, :=\, N\rtimes K(S)\, \subset\, N\rtimes S\, .
\end{equation}
In other words, $G$ is the subset $N\times K(S)$ of $N\times S$ which is in fact
a Lie subgroup of $N\rtimes S$. The group $G$
acts on $N$ using the rule given in \eqref{e1}.

The group of biholomorphisms
of the complex manifold $N$ will be denoted by $\text{Hol}(N)$. Let
\begin{equation}\label{t}
\tau\, :\, G\,\longrightarrow\, \text{Hol}(N)
\end{equation}
be the homomorphism defined by the above action.

\subsection{Equivariant hermitian principal bundles}

Let $H$ be a connected reductive linear algebraic group defined over
$\mathbb C$. Fix a maximal compact subgroup
\begin{equation}\label{e2}
K\, \subset\, H\, .
\end{equation}

\begin{definition}\label{def-1}
A \textit{hermitian structure} on a $C^\infty$ principal $H$--bundle $E_H$ over $N$ is a
$C^\infty$ reduction of structure group
$$
E_K\, \subset\, E_H
$$
to the subgroup $K$ in \eqref{e2}.
\end{definition}

Let $\mathfrak h$ be the Lie algebra of $H$.
Let $E_H$ be a $C^\infty$ principal $H$--bundle on $N$. It's adjoint vector bundle
$E_H\times^H\mathfrak h$ will be denoted by $\text{ad}(E_H)$.

Consider the Hodge type decomposition $(T^*N)\otimes_{\mathbb R}\mathbb C\,=\,
\Omega^{1,0}_N\oplus \Omega^{0,1}_N$.
The space of all connections on the principal $H$--bundle $E_H$ is an affine
space for the vector space $C^\infty(N;\, \text{ad}(E_H)\otimes
(\Omega^{1,0}_N\oplus \Omega^{0,1}_N))$. Two connections $\nabla_1$ and $\nabla_2$ on the
principal $H$--bundle $E_H$ are called \textit{equivalent} if
$$
\nabla_1-\nabla_2\,\in\, C^\infty(N;\, \text{ad}(E_H)\otimes\Omega^{1,0}_N)
\,\subset\, C^\infty(N;\, \text{ad}(E_H)\otimes
(\Omega^{1,0}_N\oplus \Omega^{0,1}_N))\, .
$$
An \textit{almost holomorphic} structure on $E_H$ is an equivalence class of connections
on $E_H$ \cite[p. 87, Proposition 2]{Ko}.

The curvature of a connection $\nabla$ on $E_H$ will be denoted by ${\mathcal K}(\nabla)$.
The component of ${\mathcal K}(\nabla)$ of Hodge type $(0\, ,2)$ will be denoted by
${\mathcal K}(\nabla)^{0,2}$. If two connections $\nabla_1$ and $\nabla_2$
on $E_H$ are equivalent,
then clearly we have ${\mathcal K}(\nabla_1)^{0,2} \,=\, {\mathcal K}(\nabla_2)^{0,2}$. The
almost holomorphic structure on $E_H$ defined by a connection $\nabla$ on $E_H$ is
integrable if and only if
\begin{equation}\label{ic}
{\mathcal K}(\nabla)^{0,2} \,=\,0
\end{equation}
(see \cite[p. 87, Proposition 3]{Ko}). An integrable almost holomorphic
structure on $E_H$ is a holomorphic structure on the principal $H$--bundle $E_H$.

An \textit{almost holomorphic hermitian principal} $H$--bundle over $N$
is a triple $(E_H\, ,\nabla\, ,E_K)$, where $(E_H\, ,\nabla)$ is an almost holomorphic
principal $H$--bundle over $N$, and $E_K\, \subset\, E_H$ is a hermitian
structure on $E_H$.

We will often suppress the notation for the equivalence class of
connections; so when we say that $E_H$ is an almost holomorphic principal $H$--bundle
we mean that $E_H$ is equipped with an equivalence class of connections.

An \textit{isomorphism} from an almost holomorphic hermitian principal $H$--bundle
$(E_H\, ,E_K)$ to an almost holomorphic hermitian principal $H$--bundle $(E'_H\, ,
E'_K)$ is a $C^\infty$ isomorphism of principal $H$--bundles
$$
\begin{matrix}
E_H & \stackrel{f_0}{\longrightarrow} & E'_H\\
\Big\downarrow && \Big\downarrow \\
N & \stackrel{\rm Id}{\longrightarrow} & N
\end{matrix}
$$
such that
\begin{itemize}
\item $f_0$ takes the almost holomorphic structure on $E_H$ to that of
$E'_H$, and

\item $f_0(E_K)\,=\, E'_K$.
\end{itemize}

An almost holomorphic hermitian principal $H$--bundle over $N$ whose almost
complex structure is integrable is called a \textit{holomorphic hermitian
principal} $H$--bundle. An \textit{isomorphism} between two holomorphic hermitian
principal $H$--bundles is an isomorphisms of the underlying
almost holomorphic hermitian principal $H$--bundles.

Now we consider the action of $G$ on $N$ (see \eqref{h1} and \eqref{e1}).

\begin{definition}\label{def1}
An \textit{equivariant} hermitian principal $H$--bundle
over $N$ is a triple of the form $(E_H\, ,E_K\, ;\rho)$, where
$f\, :\, E_H\, \longrightarrow\, N$ is a $C^\infty$ principal
$H$--bundle, $E_K\, \subset\, E_H$ is a hermitian structure, and
$$
\rho\, :\, G\times E_H\, \longrightarrow\, E_H
$$
is a $C^\infty$ left--action of the group $G$ on $E_H$, satisfying the
following conditions:
\begin{enumerate}
\item $f\circ \rho (g\, ,z)\, =\, \tau(g)(f(z))$ for all $z\,\in\, E_H$ and
$g\, \in\, G$, where $\tau$ is defined in \eqref{t},

\item the actions of $G$ and $H$ on $E_H$ commute, and

\item $\rho(G\times E_K)\,=\, E_K$.
\end{enumerate}
\end{definition}

The first two of the above three conditions mean that $z\, \longmapsto\, \rho (g\, ,z)$
is a $C^\infty$ isomorphism of the principal $H$--bundle $E_H$ with the pulled
back principal $H$--bundle $\tau(g^{-1})^*E_H$. The last condition implies that
this isomorphism between $E_H$ and $\tau(g^{-1})^*E_H$ takes $E_K$ to
$\tau(g^{-1})^*E_K$.

An \textit{isomorphism} between two equivariant hermitian principal $H$--bundles
$(E_H\, ,E_K\, ;\rho)$ and $(E'_H\, ,E'_K\, ;\rho')$ over $N$ is
a $C^\infty$ isomorphism of principal $H$--bundles
$$
f_0\, :\, E_H\, \longrightarrow\, E'_H
$$
such that $f_0(E_K)\,=\, E'_K$, and $f_0(\rho(g\, , z))\,=\, \rho'(g\, , f_0(z))$ for
all $g\,\in\, G$ and $z\,\in\, E_H$.

Let $(E_H\, ,E_K\, ;\rho)$ be an equivariant hermitian principal $H$--bundle
over $N$. A connection $\nabla$ on the principal $K$--bundle $E_K$ is called
{\it invariant} if the action of $G$ on $E_K$ given by $\rho$ preserves $\nabla$.
In other words, $\nabla$ is invariant if and only if for every $g\, \in\, G$,
the isomorphism of $E_K$ with $\tau(g^{-1})^*E_K$ defined by
$z\, \longmapsto\, \rho (g\, ,z)$ takes $\nabla$ to the pulled back connection
$\tau(g^{-1})^*\nabla$ on $\tau(g^{-1})^*E_K$.

\begin{definition}\label{def-2}
An \textit{equivariant} almost holomorphic hermitian principal $H$--bundle over $N$
is an equivariant hermitian principal $H$--bundle $(E_H\, ,E_K\, ;\rho)$ such that $E_H$
is equipped with an almost holomorphic structure satisfying the following condition: for
each $g\, \in\, G$, the diffeomorphism of $E_H$ defined by
$z\, \longmapsto\, \rho (g\, ,z)$ preserves the almost complex structure on $E_H$.
\end{definition}

An \textit{isomorphism} between two equivariant almost
holomorphic hermitian principal $H$--bundles
$(E_H\, ,E_K\, ;\rho)$ and $(E'_H\, ,E'_K\, ;\rho')$ over $N$ is
an isomorphism of equivariant hermitian principal $H$--bundles
$$
f_0\, :\, E_H\, \longrightarrow\, E'_H
$$
that takes the almost complex structure on $E_H$ to that on $E'_H$.

\begin{lemma}\label{lem1}
Let $(E_H\, ,E_K\, ;\rho)$ be an equivariant almost holomorphic hermitian
principal $H$--bundle on $N$. Then there is a unique connection $\nabla^K$
on the principal $K$--bundle $E_K$ satisfying the following two conditions:
\begin{enumerate}
\item The connection $\nabla^H$ on $E_H$ induced by $\nabla^K$ lies in the
equivalence class defining the almost holomorphic structure on $E_H$, and

\item the connection $\nabla^K$ is invariant.
\end{enumerate}
\end{lemma}

\begin{proof}
Since $H$ is complex reductive, there is a unique connection $\nabla^K$
on the principal $K$--bundle $E_K$ such that the connection $\nabla^H$ on $E_H$ induced
by $\nabla^K$ lies in the equivalence class defining the almost holomorphic structure
on $E_H$ \cite[pp. 191--192, Proposition 5]{At}. From the uniqueness of the connection
$\nabla^K$ it follows immediately that it is invariant.
\end{proof}

\section{Equivariant almost holomorphic hermitian principal bundles}\label{sec3}

Consider $G$ defined in \eqref{h1}.
We note that $G/K(S) \,=\, (N\rtimes S)/S\,=\, N$. Therefore, the quotient map
\begin{equation}\label{pb2}
G\,\longrightarrow\, G/K(S) \,=\, N
\end{equation}
defines a reduction of structure group of the principal $S$--bundle in \eqref{pb} to
the subgroup $K(S)\,\subset\, S$. For any $g\,\in\, N\rtimes S$, the leaf of the
foliation ${\mathcal H}$ in \eqref{cH} passing through the point $g$ is $gN$. Therefore,
if $g\,\in\, G$, then the leaf passing through $g$ is contained in $G$. This implies
that the connection $\nabla^S$ constructed in Section \ref{sec.conn}
produces a connection on the principal $K(S)$--bundle
$G\,\longrightarrow\, N$ in \eqref{pb2}. This induced connection on the
$K(S)$--bundle is flat because $\nabla^S$ is so.

Since the principal $K(S)$--bundle in \eqref{pb2} is a reduction of structure
group of the holomorphic principal $S$--bundle in \eqref{pb} to the maximal compact subgroup
$K(S)\,\subset\, S$, and $S$ is reductive, there is a unique connection
$\nabla^{K(S)}$ on the principal $K(S)$--bundle in \eqref{pb2}
such that the connection on the holomorphic principal $S$--bundle in \eqref{pb} induced by
$\nabla^{K(S)}$ gives the almost holomorphic structure of it
\cite[pp. 191--192, Proposition 5]{At}.

\begin{lemma}\label{lem0}
The above connection $\nabla^{K(S)}$ coincides with the flat connection
on the principal $K(S)$--bundle $G\,\longrightarrow\, N$ induced by $\nabla^S$.
In particular, $\nabla^{K(S)}$ is flat.
\end{lemma}

\begin{proof}
Since $\nabla^S$ induces a connection on the principal $K(S)$--bundle in \eqref{pb2},
and $\nabla^S$ is holomorphic, in particular it is complex, the lemma follows from 
the uniqueness of $\nabla^{K(S)}$.
\end{proof}

The Lie algebra of $K$ will be denoted by $\mathfrak k$. Let
\begin{equation}\label{cw}
{\mathcal W}\, :=\, \text{Hom}_{\mathbb R}({\mathfrak n}\, ,{\mathfrak k})
\end{equation}
be the space of all $\mathbb R$--linear homomorphisms from the vector space
$\mathfrak n$ to the vector space $\mathfrak k$. We emphasize that the
elements of $\mathcal W$ need not be Lie algebra homomorphisms.

The action of $K(S)$ on $N$ (see \eqref{eta}) produces an action of
$K(S)$ on ${\mathfrak n}$. Given a homomorphism $K(S)\,\longrightarrow\, K$ (see
\eqref{e2} for $K$), the adjoint action of $K$ on $\mathfrak k$ produces an action of
$K(S)$ on $\mathfrak k$. Therefore, given a homomorphism $K(S)\,\longrightarrow\,
K$, combining the actions of $K(S)$ on ${\mathfrak n}$ and $\mathfrak k$,
we get an action of $K(S)$ on ${\mathcal W}$ defined in \eqref{cw}.

Consider all pairs
\begin{equation}\label{b0}
(\beta\, ,\omega)\, ,
\end{equation}
where
\begin{itemize}
\item $\beta\, :\, K(S)\,\longrightarrow\, K$ is a homomorphism, and

\item $\omega\,\in\, {\mathcal W}^{K(S)}\,\subset\, \mathcal W$.
\end{itemize}
The action of $K(S)$ on $\mathcal W$ is
constructed as above, and ${\mathcal W}^{K(S)}$ is the
space of invariants. Two such pairs $(\beta\, ,\omega)$ and $(\beta'\, ,\omega')$
are called {\it equivalent} if there is an element $k\,\in\, K$ such that
\begin{itemize}
\item $\beta'(g)\,=\, k\beta(g)k^{-1}$ for all $g\,\in\, K(S)$, and

\item $\omega'(v)\, =\, \text{Ad}(k)((\omega)(v))$ for all $v\,\in\,
\mathfrak n$, where
\begin{equation}\label{ad}
\text{Ad}(k)\, :\, {\mathfrak k}\,\longrightarrow\,{\mathfrak k}
\end{equation}
is the automorphism corresponding to the automorphism of $K$ defined by $x\, \longmapsto\,
kxk^{-1}$.
\end{itemize}
Let
\begin{equation}\label{c}
\mathcal C
\end{equation}
denote the set of equivalence classes of pairs $(\beta\, ,\omega)$ of the above type.

Consider all quadruples of the form $((E_H\, ,E_K\, ;\rho)\, ,\nabla^K)$,
where $(E_H\, ,E_K\, ;\rho)$ is an equivariant hermitian principal $H$--bundle
on $N$, and $\nabla^K$ is an invariant connection on $E_K$. Two such objects
$((E_H\, ,E_K\, ;\rho)\, ,\nabla^K)$ and $((E'_H\, ,E'_K\, ;\rho')\, ,\nabla')$
are called {\it isomorphic} if there is an isomorphism of equivariant hermitian
principal $H$--bundles between $(E_H\, ,E_K\, ;\rho)$ and $(E'_H\, ,E'_K\, ;\rho')$
that takes the connection $\nabla$ to $\nabla'$.

\begin{proposition}\label{prop1}
The set of isomorphism classes of quadruples $((E_H\, ,E_K\, ;\rho)\, ,\nabla^K)$
of the above type is in bijection with the set $\mathcal C$ in \eqref{c}.
\end{proposition}

\begin{proof}
Consider the flat connection $\nabla^{K(S)}$ on the principal $K(S)$--bundle $G\,
\longrightarrow\, N$ (see Lemma \ref{lem0}). The left-translation action of $G$ on
itself clearly preserves this connection. Indeed, this follows immediately from the
fact that the distribution $\mathcal H$ in \eqref{cH} is preserved by the
left-translation action of $N\rtimes S$ on itself. Note that by Lemma \ref{lem0},
the connection $\nabla^{K(S)}$ coincides with the one induced by the connection
$\nabla^S$ that $\mathcal H$ defines.

Take any pair $(\beta\, ,\omega)$ as in \eqref{b0}. Let $E_K$ be the
principal $K$--bundle on $N$ obtained by extending the structure group of
the principal $K(S)$--bundle $G\, \longrightarrow\, N$ using the homomorphism $\beta$.
The connection $\nabla^{K(S)}$ on the principal $K(S)$--bundle $G\,
\longrightarrow\, N$ induces a connection on the associated bundle $E_K$.
This induced connection on the principal $K$--bundle $E_K$ will be denoted by
$\nabla'$. Note that the total space of $E_K$ is the quotient of $G\times K$
where two points $(g_1\, ,k_1)$ and $(g_2\, ,k_2)$ are identified if there is
an element $g\,\in\, K(S)$ such that $g_2\,=\, g_1g$ and $k_2\,=\,
\beta(g)^{-1}k_1$. Therefore, the left--translation action of $G$ on
$G\times K$ descends to an action of $G$ on the quotient space $E_K$. This
action of $G$ on $E_K$ clearly commutes with the action of $K$ on
the principal $K$--bundle $E_K$. Also, the action of $G$ on $E_K$ preserves the
connection $\nabla'$ because the action of $G$ on the the principal $K(S)$--bundle
$G\, \longrightarrow\, N$ preserves the connection $\nabla^{K(S)}$.

Let $\text{ad}(E_K)\,=\, E_K\times^K {\mathfrak k}$ be the adjoint vector
bundle for $E_K$. The action of $G$ on $E_K$ defines an action on
$\text{ad}(E_K)$. This action of $G$ on $E_K$ and the action of $G$ on $N$
together produce an action of $G$ on the vector bundle $\text{ad}(E_K)\otimes
T^*N$, where $T^*N$ is the real cotangent vector bundle on $N$. Consider
fiber $(\text{ad}(E_K)\otimes T^*N)_{e_N}$ of $\text{ad}(E_K)\otimes T^*N$
over the identity element $e_N$. We will show that it is canonically identified
with $\text{Hom}_{\mathbb R}({\mathfrak n}\, ,{\mathfrak k})$.

The fiber $T^*_{e_N}N$ is identified with
${\mathfrak n}^*$. The fiber $(E_K)_{e_N}$ is canonically identified with
$K$ by sending any $k\, \in\, K$ to the equivalence class of $(e_N\, ,k)$
(recall that $E_K$ is a quotient of $G\times K$). The identification
between $(E_K)_{e_N}$ and $K$ produces an isomorphism between
$\text{ad}(E_K)_{e_N}$ and $\mathfrak k$ by sending any $v\, \in\,
\mathfrak k$ to the equivalence class of $(e_K\, ,v)\,\in\, 
K\times \mathfrak k$, where $e_K$ is the identity element of $K$;
the vector bundle $\text{ad}(E_K)$ is a quotient of $E_K\times\mathfrak k$,
and using the identification of $K$ with $(E_K)_{e_N}$, the identity element
$e_K$ gives an element of $(E_K)_{e_N}$. Therefore, the fiber
$(\text{ad}(E_K)\otimes T^*N)_{e_N}$ is canonically identified
with ${\mathfrak n}^*\otimes {\mathfrak k}\,=\,
\text{Hom}_{\mathbb R}({\mathfrak n}\, ,{\mathfrak k})$.

Using the above identification of $\text{Hom}_{\mathbb R}({\mathfrak n}\, ,
{\mathfrak k})$ with $(\text{ad}(E_K)\otimes T^*N)_{e_N}$, the element
$\omega\, \in\, \text{Hom}_{\mathbb R}({\mathfrak n}\, ,
{\mathfrak k})$ gives an element of $(\text{ad}(E_K)\otimes T^*N)_{e_N}$. Since
$\omega$ is fixed by the action of $K(S)$, there is a unique
$G$--invariant $C^\infty$ section $\widetilde\omega$ of $\text{ad}(E_K)\otimes
T^*N$ such that
$$
{\widetilde\omega}(e_N)\,=\, \omega\, .
$$
Consider the connection $\nabla'+{\widetilde\omega}$ on $E_K$. Since both
$\nabla'$ and $\widetilde\omega$ are preserved by the action of $G$, it follows
immediately the connection $\nabla'+{\widetilde\omega}$ is also
preserved by the action of $G$.

Let $E_H$ be the principal $H$--bundle on $N$ obtained by extending the
structure group of the principal $K$--bundle $E_K$ using the inclusion of
$K$ in $H$. Consider the connection on $E_H$ induced by $\nabla'+{\widetilde\omega}$.
The $(0\, ,1)$--type component of it produces a holomorphic structure on $E_H$
because the connection comes from a connection on $E_K$. The action of $G$ on
$E_K$ produces an action of $G$ on the associated by bundle $E_G$. This
action of $E_G$ will be denoted by $\rho$.

Therefore, the quadruple $((E_H\, ,E_K\, ;\rho)\, ,\nabla'+{\widetilde\omega})$
satisfies all the required conditions.

To construct in inverse map, take any quadruple $((E_H\, ,E_K\, ;\rho)\, ,\nabla^K)$
as in the lemma. Fix a point
$$
z_0\, \in\, (E_K)_{e_N}
$$
in the fiber over the identity element $e_N$. Let
$$
\beta\, :\, K(S)\, \longrightarrow\, K
$$
be the map defined by $\rho(g\, , z_0)\,=\, z_0\beta(g)$, $g\, \in\, K(S)$.
We have
$$
z_0\beta(gh)\,=\,\rho(gh\, , z_0)\,=\,\rho(g\, , \rho(h\, ,z_0)) \,=\,
\rho(g\, , z_0\beta(h))
$$
$$
\,=\,\rho(g\, , z_0)\beta(h)\,=\, z_0\beta(g)\beta(h)\, .
$$
Therefore, $\beta$ is a homomorphism.

Let $d\, :\, T_{z_0}E_K\, \longrightarrow\, T_{e_N} N\,=\, \mathfrak n$
be the differential, at $z_0$, of the natural projection of $E_K$ to $N$.
Let ${\mathcal H}_{z_0}\, \subset\, T_{z_0}E_K$ be the horizontal subspace
for the connection $\nabla^K$ on $E_K$. The restriction of the homomorphism
$d$ to ${\mathcal H}_{z_0}$ is an isomorphism. Let
$$
d'\, :\, {\mathfrak n}\, \stackrel{\sim}{\longrightarrow}\,
{\mathcal H}_{z_0}\, \subset\, T_{z_0}E_K
$$
be the inverse of $d\vert_{{\mathcal H}_{z_0}}$. We note that the action of $N$ on
$E_K$ given by $\rho$ produces a homomorphism
$$
\delta\, :\,  {\mathfrak n}\, \longrightarrow\,T_{z_0}E_K\, .
$$
Now note that the image of the homomorphism
$$
d'-\delta\, :\, {\mathfrak n}\,\longrightarrow\,T_{z_0}E_K
$$
lies in the kernel of the homomorphism $d$. The kernel of $d$ is identified
with the Lie algebra $\mathfrak k$ because $(E_K)_{e_N}$ is an orbit of
the free action of $K$ on $E_K$. Therefore, we have
$$
d'-\delta\, :\, {\mathfrak n}\,\longrightarrow\, {\mathfrak k}
$$
In other words,
$$
\omega\, :=\, d'-\delta\, \in\, \text{Hom}_{\mathbb R}({\mathfrak n}\, ,
{\mathfrak k})\, .
$$
Since the connection $\nabla^K$ is preserved by the action of $G$, it
follows that $\omega$ is fixed by the action of $K(S)$.

Therefore, $(\beta\, ,\omega)\, \in\, {\mathcal C}$.

It is straightforward to check that the above two constructions are inverses of
each other.
\end{proof}

\begin{lemma}\label{lem2}
The set of isomorphism classes of equivariant almost holomorphic hermitian
principal $H$--bundle over $N$ is in bijection with $\mathcal C$ in \eqref{c}.
\end{lemma}

\begin{proof}
Using Lemma \ref{lem1}, the set of isomorphism classes of equivariant almost holomorphic
hermitian principal $H$--bundle over $N$ is identified with set of the isomorphism classes
of quadruples $((E_H\, ,E_K\, ;\rho)\, ,\nabla^K)$ in Proposition \ref{prop1}. Therefore,
the lemma follows from Proposition \ref{prop1}.
\end{proof}

As before, the Lie algebra of $H$ will be denoted by
$\mathfrak h$. Since $K$ is a maximal compact subgroup
of the complex reductive group $H$, the inclusion of $\mathfrak k$ in $\mathfrak h$ produces
a $\mathbb C$--linear
isomorphism of ${\mathfrak k}\otimes_{\mathbb R}{\mathbb C}$ with $\mathfrak h$.
Let $\overline{\mathfrak n}$ be the complex vector space conjugate to $\mathfrak n$.
So the underlying real vector space for $\overline{\mathfrak n}$ is the
underlying real vector space for $\mathfrak n$, but the multiplication by $\lambda\,\in\,
\mathbb C$ on $\overline{\mathfrak n}$ is the multiplication by $\overline{\lambda}$
on $\mathfrak n$. Therefore, $\mathcal W$ (see \eqref{cw}) has the following natural
identification
\begin{equation}\label{ii}
{\mathcal W}\,=\,
\text{Hom}_{\mathbb C}(\overline{\mathfrak n}\, ,{\mathfrak h})\,=\,
{\mathfrak h}\otimes_{\mathbb C}\overline{\mathfrak n}^*\, .
\end{equation}
This identification commutes with the actions of $G$.

\section{Equivariant holomorphic hermitian principal bundles}

Now assume that the connected component of the center of $K(S)$ containing the
identity element is nontrivial. Fix a subgroup
\begin{equation}\label{Z}
Z\,=\, \text{U}(1) \,\subset\, K(S)
\end{equation}
contained in the center of $K(S)$.

Henceforth, we assume that there is a nontrivial character
\begin{equation}\label{chi0}
\chi_0\, :\, Z\,\longrightarrow\, {\mathbb C}^*
\end{equation}
such that the action of any $g\,\in\, Z$ on the Lie algebra ${\mathfrak n}$, given by
$\eta$ in \eqref{eta}, is multiplication by $\chi_0(g)$.

Consider $\mathcal W\,=\, \text{Hom}_{\mathbb C}(\overline{\mathfrak n}\, ,{\mathfrak h})$
defined in \eqref{cw} (see \eqref{ii}). Any $\mathbb C$--linear map
$$
\alpha\, :\, \overline{\mathfrak n}\,\longrightarrow\,\mathfrak h
$$
produces a linear map $\bigwedge^2 \alpha\, :\, \bigwedge^2 \overline{\mathfrak n}\,
\longrightarrow\,\bigwedge^2 \mathfrak h$. Composing $\bigwedge^2 \alpha$ with the
Lie bracket $\bigwedge^2 \mathfrak h\, \longrightarrow\, \mathfrak h$, we get
a $\mathbb C$--linear map
\begin{equation}\label{f}
\varphi(\alpha)\, :\, \bigwedge\nolimits^2 \overline{\mathfrak n}\,\longrightarrow\,
\mathfrak h\, .
\end{equation}
Define the subset of $\mathcal C$ (see \eqref{c})
\begin{equation}\label{c0}
{\mathcal C}_0\, :=\, \{(\beta\, ,\omega)\,\in\, {\mathcal C}\,\mid\,
\varphi(\omega)\,=\, 0\}\,\subset\, \mathcal C\, ,
\end{equation}
where the map $\varphi$ is constructed in \eqref{f}.

\begin{theorem}\label{thm1}
The set of isomorphism classes of equivariant holomorphic hermitian principal
$H$--bundle over $N$ is in bijection with ${\mathcal C}_0$ defined in \eqref{c0}.
\end{theorem}

\begin{proof}
Take a homomorphism
\begin{equation}\label{be}
\beta\, :\, K(S)\,\longrightarrow\, K\, .
\end{equation}
Consider the principal $K(S)$--bundle $G\,\longrightarrow\, N$ in \eqref{pb2}. Let
$$
E_K\, :=\, G\times^{K(S)} K\, \longrightarrow\, N
$$
be the principal $K$--bundle obtained by extending the structure group of it
using the homomorphism $\beta$ in \eqref{be}.
Consider the connection $\nabla^{K(S)}$ on the principal $K(S)$--bundle in
\eqref{pb2} (see Lemma \ref{lem0}). It induces a connection on the above associated
principal $K$--bundle $E_K$. This induced connection on $E_K$ will be denoted by
$\nabla^K$. This connection $\nabla^K$ is flat because $\nabla^{K(S)}$ is so.

The left--translation action of $G$ on itself produces a left--action of $G$ on
$E_K$. To see this action, first note that $E_K$ is the quotient of $G\times K$
where two points $(g_1\, ,k_1)$ and $(g_2\, ,k_2)$ are identified if there is an
element $z\,\in\, K(S)$ such that $g_2\,=\, g_1z$ and $k_2\,=\, \beta(z)^{-1}k_1$.
Therefore, the left--translation action of $G$ on $G\times K$ descends to a
left--action of $G$ on $E_K$.
The above connection $\nabla^{K}$ on $E_K$ is preserved by this action of $G$ on $E_K$.
Indeed, this follows immediately from the fact that the connection $\nabla^{K(S)}$
is preserved by the left--translation action of $G$ on itself (see the proof
of Proposition \ref{prop1}).

Let $E_H\, :=\, E_K\times^K H\, \longrightarrow\, N$ be the principal $H$--bundle
obtained by extending the structure group of the principal $K$--bundle $E_K$ using
the inclusion of $K$ in $H$. Note that $E_H$ is identified with the principal $H$--bundle
$$
G\times^{K(S)} H\, \longrightarrow\, N
$$
obtained by extending the structure group of the principal $K(S)$--bundle
$G\, \longrightarrow\, N$ using the composition homomorphism
$$
K(S)\,\stackrel{\beta}{\longrightarrow}\, K\, \hookrightarrow\, H\, .
$$
The left--action of $G$ on $E_K$ produces a left--action of $G$ on
$E_H$. This action of $G$ on $E_H$ will be denoted by $\rho$.
The connection $\nabla^K$ on $E_K$ induces a connection on the associated bundle $E_H$.
This induced connection on the principal $H$--bundle $E_H$ will be denoted by
\begin{equation}\label{nh}
\nabla^H\, .
\end{equation}
This connection $\nabla^H$ is flat because $\nabla^K$ is so. The flat connection
$\nabla^H$ defines a holomorphic structure on the principal $H$--bundle $E_H$. We note
that $\nabla^{H}$ is preserved by the above defined action $\rho$ of $G$ on $E_H$
because $\nabla^{K(S)}$ is preserved by the left--translation action of $G$ on itself.
This implies that for each $g\,\in\, G$, the diffeomorphism of $E_H$ given by the
action of $g$ is holomorphic.

Consequently,
\begin{equation}\label{tr}
(E_H\, ,E_K\, ; \rho)
\end{equation}
is an equivariant holomorphic hermitian principal
$H$--bundle on $N$. It corresponds to the pair $(\beta\, ,0)\,\in\, {\mathcal C}$ by
the bijection in Lemma \ref{lem2}.

Now take an invariant element
\begin{equation}\label{omega}
\omega\, \in\, {\mathcal W}^{K(S)}
\end{equation}
(see \eqref{cw} and \eqref{ii})
for the action of $K(S)$ on $\mathcal W$ constructed using $\beta$ in \eqref{be}. Let
$$
\text{ad}(E_H)\,=\, E_H\times^H {\mathfrak h}\,\longrightarrow\, N
$$
be the adjoint vector bundle for the principal $H$--bundle $E_H$ in \eqref{tr}. The fiber
$\text{ad}(E_H)_{e_N}$ is canonically identified
with the Lie algebra $\mathfrak h$ (as before, $e_N$ is the identity element of $N$).
To see this identification, first note that $\text{ad}(E_H)$ is the quotient of
$G\times H\times {\mathfrak h}$ where two elements $(g_1\, ,h_1\, ,v_1)$ and
$(g_2\, ,h_2\, ,v_2)$ of $G\times H\times {\mathfrak h}$ are identified if there are
elements $x\, \in\, K(S)$ and $h\,\in\, H$ such that
$g_2\,=\, g_1x^{-1}$, $h_2\,=\, \beta(x)h_1h^{-1}$ and $v_2\,=\, \text{Ad}(h)(v_1)$
(see \eqref{ad} for $\text{Ad}(h)$). The Lie algebra $\mathfrak h$ is
identified with the fiber $\text{ad}(E_H)_{e_N}$ by sending any $v\,\in\,
{\mathfrak h}$ to the equivalence class of $(e\, , e_H\, , v)$, where $e$ and $e_H$
are the identity elements of $G$ and $H$ respectively.

Since the holomorphic tangent space $T^{1,0}_{e_N}N$ to $N$ at $e_N$ is identified with
$\mathfrak n$, the anti-holomorphic tangent space $T^{0,1}_{e_N}N$ is identified with
$\overline{\mathfrak n}$. in view of the above isomorphism of $\mathfrak h$ with
$\text{ad}(E_H)_{e_N}$, the vector space ${\mathcal W}$ in \eqref{ii} gets identified
with the space of $\mathbb C$--linear maps $\text{Hom}_{\mathbb C}(T^{0,1}_{e_N}N\, ,
\text{ad}(E_H)_{e_N})$.

The action $\rho$ of $G$ on $E_H$ in \eqref{tr} produces an action of $G$ on the adjoint
vector bundle $\text{ad}(E_H)$. Therefore, we get an action of the isotropy subgroup $K(S)$
on the fiber $\text{ad}(E_H)_{e_N}$. The above
identification between ${\mathcal W}$ and $\text{Hom}_{\mathbb C}(T^{0,1}_{e_N}N\, ,
\text{ad}(E_H)_{e_N})$ clearly intertwines the actions of $K(S)$.

Let
$$
\omega'\, \in\, \text{Hom}_{\mathbb C}(T^{0,1}_{e_N}N\, ,
\text{ad}(E_H)_{e_N}) \,=\, \text{ad}(E_H)_{e_N}\otimes\Omega^{0,1}_{N,e_N}
\,=\, \text{ad}(E_H)_{e_N}\otimes\overline{\mathfrak n}^*
$$
be the element that corresponds to $\omega$ in \eqref{omega} by the above identification
between ${\mathcal W}$ and $\text{Hom}_{\mathbb C}(T^{0,1}_{e_N}N\, ,
\text{ad}(E_H)_{e_N})$.
Since $\omega$ is fixed by the action of $K(S)$, this element $\omega'$ is also fixed
by the action of $K(S)$. Therefore, there is a unique $G$--invariant section
\begin{equation}\label{wiom}
\widetilde{\omega}\,\in\, C^\infty(N;\, \text{ad}(E_K)\otimes \Omega^{0,1}_N)^G
\end{equation}
such that $\widetilde{\omega}(e_N)\,=\, \omega'$.

Consider the connection $\nabla^H$ on $E_H$ constructed in \eqref{nh}. Note that
$$
\widetilde{\nabla}^H\, :=\, \nabla^H+\widetilde{\omega}
$$
is a connection on $E^H$. Let $\widetilde{E}_H$ be the almost homomorphic principal
$H$--bundle defined by this connection $\widetilde{\nabla}^H$. Therefore,
$$
(\widetilde{E}_H\, ,E_K\, ; \rho)
$$
is an equivariant almost holomorphic hermitian principal
$H$--bundle on $N$, where $E_K$ are $\rho$ are as in \eqref{tr}. This
equivariant almost holomorphic hermitian principal $H$--bundle corresponds to the
pair $(\beta\, ,\omega)$ by the bijection in Lemma \ref{lem2}.

Let ${\mathcal K}(\widetilde{\nabla}^H)$ denote the curvature of the above connection
$\widetilde{\nabla}^H$ on $E_H$. The component of ${\mathcal K}(\widetilde{\nabla}^H)$
of Hodge type $(0\, ,2)$ will be denoted by ${\mathcal K}(\widetilde{\nabla}^H)^{0,2}$. We
note that the above almost holomorphic principal $H$--bundle $\widetilde{E}_H$ is
holomorphic if and only if
\begin{equation}\label{cc}
{\mathcal K}(\widetilde{\nabla}^H)^{0,2}\,=\, 0
\end{equation}
(see \eqref{ic}).

Let $\widehat{\nabla}$ be the connection on the adjoint vector bundle
$\text{ad}(E_H)$ induced by the connection $\nabla^H$ on $E_H$ in \eqref{nh}.
Since the connection $\nabla^H$ is flat, we have
\begin{equation}\label{wsh2}
{\mathcal K}(\widetilde{\nabla}^H)^{0,2}\,=\, \widehat{\nabla}(\widetilde{\omega})^{0,2}
+ (\widetilde{\omega}\bigwedge \widetilde{\omega})^{0,2}\, ,
\end{equation}
where the superscript $(0\, ,2)$ denotes the component of Hodge type $(0\, ,2)$.

We will show that
\begin{equation}\label{wsh}
\widehat{\nabla}(\widetilde{\omega})^{0,2}\,=\, 0\, .
\end{equation}

Let $Z^*\,:=\, \text{Hom}(Z\, ,{\mathbb C}\setminus\{0\})$ be the group of characters of
the subgroup $Z$ in \eqref{Z}. Note that
$Z^*$ is isomorphic to $\mathbb Z$. Consider the action of the isotropy subgroup $K(S)$
on the fiber $\text{ad}(E_H)_{e_N}\,=\, \mathfrak h$ given by $\rho$ (the action of
$K(S)$ on $\mathfrak h$ is given by the homomorphism $\beta$). Restrict this action of
$K(S)$ to the subgroup $Z\,\subset\, K(S)$. Let
\begin{equation}\label{de}
\text{ad}(E_H)_{e_N}\,=\, \bigoplus_{\chi\in Z^*} V^\chi
\end{equation}
be the isotypical decomposition of the $Z$--module. Since $Z$ is contained in the center of
$K(S)$, the action of $K(S)$ on $\text{ad}(E_H)_{e_N}$ preserves the decomposition in
\eqref{de}.

Take any $\chi\,\in\, Z^*$. Since the subspace $V^\chi\,\subset\, \text{ad}(E_H)_{e_N}$
is preserved by the action of the isotropy subgroup $K(S)\, \subset\, G$, there is a
unique $C^\infty$ subbundle
$$
{\mathcal V}^\chi\, \subset\, \text{ad}(E_H)
$$
such that
\begin{itemize}
\item ${\mathcal V}^\chi$ is preserved by the action of $G$ on $\text{ad}(E_H)$, and

\item the fiber ${\mathcal V}^\chi_{e_N}\,=\, V^\chi$.
\end{itemize}

Since ${\mathcal V}^\chi$ is preserved by the action of $G$, it follows that
${\mathcal V}^\chi$ is preserved by the connection $\widehat{\nabla}$ (as before,
$\widehat{\nabla}$ is the connection of $\text{ad}(E_H)$ induced by the connection
$\nabla^H$ on $E_H$). Indeed, the condition that ${\mathcal V}^\chi$ is
preserved by the action of $G$ implies that ${\mathcal V}^\chi$ is identified with the
flat vector bundle $G\times^{K(S)} V^\chi$ associated to the flat principal $K(S)$--bundle
$G\,\longrightarrow\, N$ for the $K(S)$--module ${\mathcal V}^\chi_{e_N}\,=\,V^\chi$.

We recall that $\widetilde{\omega}(e_N)\,=\, \omega'\, \in\,
\text{ad}(E_H)_{e_N}\otimes\Omega^{0,1}_{N,e_N}$ is fixed by the action of the isotropy
subgroup $K(S)$. In particular, $\widetilde{\omega}(e_N)$ is fixed by the
action of $Z\,\subset\, K(S)$. The group $Z$ acts on the complex vector space
$\Omega^{0,1}_{N,e_N}\,=\, \overline{\mathfrak n}^*$ as multiplication through the
character $\chi_0$ in \eqref{chi0}. Consequently, we have
$$
\omega'\, \in\, V^{\chi^{-1}_0}\otimes \Omega^{0,1}_{N,e_N}
$$
(see \eqref{de}). This implies that
$$
\widetilde{\omega}\,\in\, C^\infty(N;\, {\mathcal V}^{\chi^{-1}_0}\otimes \Omega^{0,1}_N)\, ,
$$
because $\widetilde{\omega}$ is fixed by the action of $G$ and the subbundle
${\mathcal V}^{\chi^{-1}_0}\,\subset\, \text{ad}(E_H)$ is preserved by the action of
$G$. Therefore, we have
$$
\widehat{\nabla}(\widetilde{\omega})^{0,2}\,\in\,
C^\infty(N;\, {\mathcal V}^{\chi^{-1}_0}\otimes \Omega^{0,2}_N)^G
$$
(recall that the connection $\widehat{\nabla}$ is invariant under the action $\rho$
of $G$ on $E_H$). Therefore, the evaluation
$$
\widehat{\nabla}(\widetilde{\omega})^{0,2}(e_N)\,\in\,
{\mathcal V}^{\chi^{-1}_0}\otimes \Omega^{0,2}_{N,e_N}
$$
is fixed under the action of $K(S)$, in particular, it is fixed by the action of
the subgroup $Z$. But $Z$ acts on $\Omega^{0,2}_{N,e_N}$ as multiplication through the
character $\chi^2_0$, because it acts on $\Omega^{0,1}_{N,e_N}$ as multiplication
through the character $\chi_0$. Therefore, $Z$ acts on $V^{\chi^{-1}_0}\otimes
\Omega^{0,2}_{N,e_N}$ as multiplication through the character $\chi_0$. Since
$\chi_0$ is nontrivial, this implies that we have the space of invariants
$$
(V^{\chi^{-1}_0}\otimes \Omega^{0,2}_{N,e_N})^Z\,=\, 0\, .
$$
In particular, we have $\widehat{\nabla}(\widetilde{\omega})^{0,2}(e_N)\,=\, 0$.
Therefore, $\widehat{\nabla}(\widetilde{\omega})^{0,2}\,=\, 0$, because
it is $G$--invariant. This proves \eqref{wsh}.

In view of \eqref{wsh}, from \eqref{wsh2} we conclude that
$$
{\mathcal K}(\widetilde{\nabla}^H)^{0,2}\,=\,
(\widetilde{\omega}\bigwedge \widetilde{\omega})^{0,2}\, .
$$
Now, it is easy to see that $(\widetilde{\omega}\bigwedge \widetilde{\omega})^{0,2}(e_N)\,=\,
\varphi(\omega)$, where $\varphi$ is constructed in \eqref{f}. Therefore,
we conclude that \eqref{cc} holds if and only if $\varphi(\omega)\,=\,0$.
This completes the proof.
\end{proof}

\section{Examples}

Let $G_0$ be a simple linear algebraic group defined over
$\mathbb C$. Let
$$
P\, \subsetneq\, G_0
$$
be a proper parabolic subgroup. The unipotent radical of $P$ will be denoted
by $R_u(P)$. The quotient
$$
L'(P)\,:=\, P/R_u(P)
$$
is a connected reductive complex linear algebraic group. Fix a connected
complex reductive subgroup
$$
L(P)\,\subset\, P
$$
such that the composition
\begin{equation}\label{e-1}
L(P)\,\hookrightarrow\, P\,\longrightarrow\, P/R_u(P)\,=\, L'(P)
\end{equation}
is an isomorphism. Such a subgroup $L(P)$ is called a Levi factor of $P$
\cite[p. 184]{Hu}. We note that Levi factors of $P$ exist, and any two Levi factors
of $P$ are conjugate by an element of $R_u(P)$ \cite[p. 185, Theorem]{Hu}.

The Levi subgroup $L(P)$ has the adjoint action on $R_u(P)$. The group $P$ is
identified with the corresponding semidirect product $R_u(P)\rtimes L(P)$ by sending any
$(u\, ,g)\,\in\, R_u(P)\times L(P)$ to $ug\,\in\, P$.

In the previous notation, $N\,=\, R_u(P)$ and $S\,=\, L(P)$.

Let $R_n({\mathfrak p})$ denote the Lie algebra of the
unipotent radical $R_u(P)$. Assume that $R_n({\mathfrak p})$
is abelian. Then $P$ is a maximal proper parabolic subgroup. Therefore, the
center of $L(P)$ is isomorphic to ${\mathbb C}^*$. Hence, the center of a
maximal compact subgroup $K$ of $L(P)$ is isomorphic to ${\rm U}(1)$. The adjoint
action of the center of $K$ on $R_n({\mathfrak p})$ is multiplication through a single
nontrivial character of the center.

All maximal proper parabolic subgroups of $\text{SL}(n, {\mathbb C})$
satisfy the condition that the unipotent radical is abelian. Both
$\text{Sp}(2n, {\mathbb C})$ and $\text{SO}(n, {\mathbb C})$ and also
the exceptional groups have parabolic subgroups satisfying the
the condition that the unipotent radical is abelian.


\end{document}